\documentclass[10pt]{article}

\title{A Remark on the Projective Geometry of Constant Curvature Spaces}
\author{Athanase Papadopoulos and Sumio Yamada}
\date{\today}

 \usepackage{amsmath}
 \usepackage{amssymb}
\usepackage{color}

\begin{document}
\newtheorem{theorem}{Theorem}
\newtheorem{proposition}[theorem]{Proposition}
\newtheorem{lemma}[theorem]{Lemma}
\newtheorem{corollary}[theorem]{Corollary}
\newtheorem{claim}[theorem]{Claim}
\newcommand{\weil}{Weil-Petersson }
\newcommand {\teich}{Teichm\"{u}ller }
\newcommand{\T}{{\mathcal T}}
\newcommand{\Tbar}{\overline{\mathcal T}}

\newenvironment{proof}{ {\sc {\bf Proof}}}{ {\sc {\bf q.e.d.}} \\}
\newenvironment{remark}{\noindent {\bf Remark}}{{\sc } \\}
\newenvironment{definition}{\noindent {\bf Definition}}{{\sc } }

\newcommand{\Tbdry}{\ensuremath{\partial {\cal T}}}
\newcommand{\g}{\ensuremath{\gamma}}
\newcommand{\del}{\partial}
\newcommand{\map}{\ensuremath{{\rm Map}(\Sigma)}}
\newcommand{\e}{\ensuremath{\varepsilon}}
\newcommand{\note}{\textcolor{red}}

\maketitle

\begin{abstract}
We highlight the relation between the projective geometries of $n$-dimensional Euclidean, spherical and hyperbolic spaces through the projective models of these spaces in the $n+1$-dimensional Minkowski space, using a cross ratio notion which is proper to each of the three geometries.

\bigskip



\end{abstract}

\thanks{The second author is supported in part by JSPS Grant-in-aid for Scientific Research No.24340009}

\section{Cross Ratio} \label{s:1}

The projective geometry of a Riemannian manifold $(M, g)$ is the geometry of the space of lines/unparametrized geodesics of the manifold.  For each dimension $n \geq 2$, the Euclidean space ${\mathbb R}^n$, the sphere $S^n$ and the hyperbolic space ${\mathbb H}^n$ provide a set of canonical Riemannian manifolds of constant sectional curvatures $0, 1$ and $-1$ that are distinct from each other in the Riemannian geometric sense, but one can consider the projective geometry on each of these spaces.  The most important qualitative feature of projective geometry is the notion of cross ratio, and the Riemannian distance and angle invariance under isometric transformations of  the underlying Riemannian manifold are replaced by a weaker invariance of the cross ratio under projective transformations.  The goal of this article is to highlight the relation among the projective geometries of these three space forms through the projective models of these spaces in the Minkowski space ${\mathbb R}^{n, 1}$, by using a cross ratio notion which is proper to each of the three geometries.

 We start by recalling some classical facts. In the Euclidean plane ${\mathbb R}^2$, consider four  ordered distinct lines $l_1, l_2, l_3, l_4$ that are concurrent at a point $A$ and let $l$   be a  line that intersects these four lines at points $A_1, A_2, A_3, A_4$ respectively. Then the cross ratio $[A_2, A_3, A_4, A_1]$ of the ordered quadruple $A_1, A_2, A_3, A_4$ does not depend on the choice of the line $l'$. This property expresses the fact that the cross ratio is a \emph{projectivity invariant}. 

As a matter of fact, Menelaus (Alexandria, 2nd century A.D.)  considered the above property not only the Euclidean plane, but also on the sphere, where the lines  are the spherical geodesics, which are the great circles of the sphere.  Once there is a parallel between the Euclidean geometry and the spherical geometry, it is natural to expect to have the corresponding statement for hyperbolic geometry.  We now define the cross ratio for the three geometries.  
\medskip

 \begin{definition}  Consider a geodesic line in Euclidean, hyperbolic and spherical geometry respectively, and let 
  $A_1, A_2, A_3, A_4$ be four ordered pairwise distinct points on that line.
 We define the cross ratio $[A_1, A_2, A_3, A_4]$, 
 in the Euclidean case, by:
 \[
 [A_2, A_3, A_4, A_1]_e :=\frac{A_2 A_4}{A_3 A_4}\cdot  \frac{A_3 A_1}{A_2 A_1},
 \]
 in the hyperbolic case, by:
 \[
[A_2, A_3, A_4, A_1]_h := \frac{\sinh A_2 A_4}{\sinh A_3 A_4} \cdot \frac{\sinh A_3 A_1}{\sinh A_2 A_1},
 \]
and in the spherical case, by:
 \[
[A_2, A_3, A_4, A_1]_s := \frac{\sin A_2 A_4}{\sin A_3 A_4}\cdot  \frac{\sin A_3 A_1}{\sin A_2 A_1},
 \]
where $A_i A_j$ stands for the distance between the pair of points $A_i$ and $A_j$, which is equal to the length of the line segment joining them.  (For this, we shall assume that in the case of spherical geometry the four points lie on a hemisphere;  instead, we could work in the elliptic space, that is, the quotient of the sphere by its canonical involution.)
 \end{definition}
 
We denote by $U^n$ the open upper hemisphere of $S^n$ equipped with the induced metric. Let $X$ and $X'$ belong to the set $\{\mathbb{R}^n, \mathbb{H}^n, U^n\}$. We shall say that a map $X\to X'$ is a \emph{perspectivity}, or a \emph{perspective-preserving transformation} if it preserves lines and if it preserves the cross ratio of quadruples of points on lines. (We note that these terms are classical, see e.g. Hadamard \cite{Had} or Busemann \cite{Busemann1953}. We also note that such maps arise indeed in perspective drawing.) In what follows, using well-known projective models in $\mathbb{R}^{n+1}$ of hyperbolic space $\mathbb{H}^n$ and of the sphere $S^n$, we define natural homeomorphisms between $\mathbb{R}^n$, $\mathbb{H}^n$ and the open upper hemisphere of $S^n$ which are perspective-preserving transformations. The proofs are elementary and  are based on first principles of geometry.

The authors would like to thank Norbert A'Campo for sharing his enthusiasm and ideas.
 
 \section{Projective Geometry}

Up to now, the word ``projective" is being used as a property of the incidence of lines/geodesics in the underlying space.  On the other hand, the $n$-dimensional sphere and the $n$-dimensional hyperbolic space are realized ``projectively" in $n+1$-dimensional Euclidean space as sets of  ``unit" length vectors. Namely the sphere $S^n$ is the set of unit vectors in ${\mathbb R}^{n+1}$ with respect to the Euclidean norm 
\[
x_1^2 + \cdots + x_n^2 + x_{n+1}^2 = 1
\]
and the hyperbolic space ${\mathbb H}^n$ is one of the two components the set of ``vectors of imaginary norm $i$" with $x_{n+1}>0$ in  ${\mathbb R}^{n+1}$ with respect to the Minkowski norm 
\[
x_1^2 + \cdots  + x_n^2 - x_{n+1}^2 = -1.
\]
These models of the two constant curvature spaces are called \emph{projective} for the geodesics in the curved spaces are realized as the intersection of the unit spheres with the two-dimensional subspace of ${\mathbb R}^{n+1}$ through the origin of this space.  We also note that each two-dimensional linear subspace intersects the 
hyperplane $\{ x_{n+1} =1\}$ in a line, that is, a Euclidean geodesic.  Hence each two-dimensional linear subspace of ${\mathbb R}^{n+1}$ represents a geodesic in each of the three geometries, consequently establishing the correspondence among the three incidence geometries.    

Now we present the main results.

\begin{theorem}[Spherical Case] \label{th:s}
 Let $P_s$ be the  projection map from the origin of $ {\mathbb R}^{n+1}$ sending the open upper hemisphere $U^n$ of $S^n$ onto the hyperplane $\{ x_{n+1} = 1 \} \subset {\mathbb R}^{n+1}$.  Then the projection map $P_s$ is a perspectivity. In particular it preserves the values of cross ratio; namely for a set of  four ordered pairwise distinct points
  $A_1, A_2, A_3, A_4$ aligned on a great circle in the upper hemisphere, 
  we have \[ [P_s(A_2),P_2(A_3),P_s(A_4),P_s(A_1)]_e=[A_2, A_3, A_4, A_1]_s.\]
\end{theorem}

\begin{proof} 
Let $u, v$ be the two points on the hyperplane $\{x_{n+1} =1\}$, let $P_s(u)=:[u], P_s(v)=:[v]$ be the points in $U$, and $d([u], [v])$ be the spherical distance between them. Finally let $\|x\|$ be the Euclidean norm of the vector $x \in {\mathbb R}^{n+1}$. We show that 
\[
\sin d( [u], [v] ) = \frac{\|u-v\|}{\|u\|\|v\|}.
\]
This follows from the following trigonometric relations:
\begin{eqnarray*}
\sin d( [u], [v] ) & = & \sin \Big[ \cos^{-1} \Big( \frac{u}{\|u\|} \cdot  \frac{v}{\|v\|} \Big)\Big]  \\ & = &  \sqrt{1 - \cos^2 \Big[ \cos^{-1} \Big( \frac{u}{\|u\|} \cdot  \frac{v}{\|v\|} \Big) \Big]} \\
 & = & \sqrt{1- ( \frac{u}{\|u\|} \cdot  \frac{v}{\|v\|} \Big)^2}   =  \frac{1}{\|u\| \|v\|} \sqrt{\|u\|^2 \|v\|^2 - (u \cdot v)^2}  \\
 & = &  \frac{1}{\|u\| \|v\|} \times (\mbox{the area of parallelogram spanned by $u$ and $v$}) \\
 & = & \frac{\|u-v\|}{\|u\| \|v\|}.
\end{eqnarray*} 

Now  for a set of  four ordered pairwise distinct points
  $A_1, A_2, A_3, A_4$ aligned on a great circle in the upper hemisphere, their spherical cross ratio $[A_2, A_3, A_4, A_1]_s$ is equal to the Euclidean cross ratio $[P_s(A_2), P_s(A_3), P_s(A_4), P_s(A_1)]_e$;
\[
\frac{ \sin d( [A_2], [A_4] )}{\sin d([A_3], [A_4])} \cdot \frac{ \sin d( [A_3], [A_1] )}{\sin d([A_2], [A_1])} 
= \frac{\frac{\|A_2-A_4\|}{\|A_2\| \|A_4\|}}{ \frac{\|A_3-A_4\|}{\|A_3\| \|A_4\|} } \cdot
\frac{\frac{\|A_3-A_1\|}{\|A_3\| \|A_1\|}}{ \frac{\|A_2-A_1\|}{\|A_2\| \|A_1\|} }  
= \frac{\|A_2-A_4\|}{\|A_3 - A_4\|} \cdot \frac{\|A_3-A_1\|}{\|A_2 - A_1\|}
\]

\end{proof}

\begin{theorem}[Hyperbolic Case] \label{th:h}
Let  $P_h$ be the  projection map of the hyperboloid ${\mathbb H}^n \subset {\mathbb R}^{n+1}$ from the origin of $ {\mathbb R}^{n+1}$ onto the unit disc of the hyperplane $\{ x_{n+1} = 1 \} \subset {\mathbb R}^{n+1}$.  Then the projection map $P_h$ is a perspectivity. In particular it preserves the values of cross ratio; namely for a set of  four ordered pairwise distinct points
  $A_1, A_2, A_3, A_4$ aligned on a geodesic in the upper hyperbolid, we have
  \[ [P_h(A_2),P_h(A_3),P_h(A_4),P_h(A_1)]_e=[A_2, A_3, A_4, A_1]_s.\]  
\end{theorem}

\begin{proof}
We follow the spherical case, where the sphere of the unit radius in ${\mathbb R}^{n+1}$  is replaced by the upper sheet of the sphere of radius $i$, namely the hyperboloid in ${\mathbb R}^{n, 1}$.  Let $u, v$ be the two points on the hyperplane $\{x_{0} =1\}$.  and $P_h(u)=:[u], P_h(v)=:[v]$ be the points in the hyperboloid, or, equivalently, the time-like vectors of unit (Minkowski) norm.   Denote by $d([u], [v])$ the hyperbolic length between the points. 
Also let $\|x\|$ be the Minkowski norm of the vector $x \in {\mathbb R}^{n,1}$.  We will show that 
\begin{eqnarray*}
\sinh d([u], [v]) = - \frac{\|u-v \|}{\|u\| \|v\|}. 
\end{eqnarray*}
Note that the number on the right hand side is positive, for $\|u\|, \|v\|$ are positive imaginary numbers, and $u-v$ is a purely space-like vector, on which the Minkowski norm of ${\mathbb R}^{n, 1}$ and the Euclidean norm of ${\mathbb R}^n$ coincide.  

This follows from the following trigonometric relations
\begin{eqnarray*}
\sinh d( [u], [v] ) & = & \sinh \Big[ \cosh^{-1} \Big( \frac{u}{\|u\|} \cdot  \frac{v}{\|v\|} \Big)\Big] 
 \\ & = &  \sqrt{\cosh^2 \Big[ \cosh^{-1} \Big( \frac{u}{\|u\|} \cdot  \frac{v}{\|v\|} \Big) \Big] -1 } \\
 & = & \sqrt{\Big( \frac{u}{\|u\|} \cdot  \frac{v}{\|v\|} \Big)^2 -1 }  
  =  \frac{\sqrt{-1}}{\|u\| \|v\|} \sqrt{\|u\|^2 \|v\|^2 - (u \cdot v)^2 }   \\
 & = &  \frac{ \sqrt{-1}}{\|u\| \|v\|} \sqrt{-1}\times (\mbox{the area of parallelogram spanned by $u$ and $v$}) \\
 & = & - \frac{\|u-v\|}{\|u\| \|v\|}.
\end{eqnarray*} 
The formula \[\sqrt{\|u\|^2 \|v\|^2 - (u \cdot v)^2 }  =  \sqrt{-1}\times (\hbox{the area of parallelogram spanned by } u \hbox{ and } v) \]
 is as stated in Thurston's notes (Section 2.6 \cite{Th}). Now by the same argument as in the spherical case, the hyperbolic cross ratio $[A_2, A_3, A_4, A_1]_h $ is equal to the Euclidean cross ratio $[P(A_2),P(A_3),P(A_4),P(A_1)]_e$.
\end{proof}

In \S \ref{s:1}, we  have referred to the notion of projectivity invariance of the cross ratio in Euclidean space. We extend this notion for the cross ratio in the spherical and hyperbolic spaces, by using the same definition. In the case of the sphere, we  restrict to configurations so that all the points considered are contained in an open hemishpere. We have the following:
\begin{corollary}
The spherical and hyperbolic cross ratios are projectivity invariants.
\end{corollary}

 This follows from the fact that the projection map $P_s$ and $P_h$ are both perspective-preserving transformations. 
 The classical proofs of this result in the cases of Euclidean and spherical geometry (known to Menelaus) relies on the fact that the cross ratio is completely determined by the angles among the lines/geodesic $l_i$'s at the vertex $A$. This proof can easily be done using the so-called Sine Rule.

 \begin{proposition}[Sine Rule]
  Given a triangle $ABC$ with sides $a,b,c$ opposite to the angles $A,B,C$ respectively, we have, in the case where the triangle is Euclidean, spherical, hyperbolic respectively:
\[\frac{a}{\sin A} =\frac{b}{\sin B}= \frac{c}{\sin C}, \,\,\, 
\frac{\sin a}{\sin A} =\frac{\sin b}{\sin B}= \frac{\sin c}{\sin C}, \,\,\, 
\frac{\sinh a}{\sin A} =\frac{\sinh b}{\sin B}= \frac{\sinh c}{\sin C}.\] 
\end{proposition}
For this and for other trigonometric formulae in hyperbolic trigonometry we refer the reader to \cite{ACP} where the proofs are given in a model-free setting. In such a setting the proofs in the hyperbolic and the spherical cases can be adapted from each other with very little changes.

\section{Generalized Beltrami-Klein's models of ${\mathbb H}^n$}

Given a bounded open convex set $\Omega$ in a Euclidean space, D. Hilbert  in (\cite{H} 1895) proposed a natural metric $H(x, y)$, now called the \emph{Hilbert metric}, defined for $x\not= y$ in $\Omega$ as the logarithm of the cross ratio of the quadruple $(x, y, b(x, y), b(y,x))$, where $b(x, y)$ is the point where the ray $R(x, y)$ from $x$ through $y$ hits the boundary  $\partial \Omega$ of $\Omega$.  This defines a 
metric on $\Omega$, which is Finslerian and projective.  We refer to the article \cite{PY}  for a parallel  treatment of the subject in hyperbolic and in spherical geometry.

The most prominent example of Hilbert metric is the Beltrami-Klein model of hyperbolic 
space, where the underlying convex set $\Omega$ is the unit ball in $\mathbb{R}^n$.  In fact, this special case was Hilbert's primary motivation to define the so-called Hilbert metric on an arbitrary bounded open convex set $\Omega$ in a Euclidean space.  Actually, for the Beltrami-Klein model, the 
size of the ball is irrelevant, as the Hilbert metric is invariant under homothety of the underlying 
Euclidean space, so that for each Hilbert metric $H_\rho(x, y)$ defined on the ball $B_
\rho(0)$ of radius $\rho >0$ centered at the origin, $(B_\rho, H_\rho)$ is 
isometric to the hyperbolic space ${\mathbb H}^n$.  

Immediate corollaries  of Theorems \ref{th:s} and \ref{th:h} are that there are new models of the hyperbolic space, which we call generalized Beltrami-Klein models.  We first set some notation: let $B^h_\rho$ and $B^s_\rho$ be the geodesic balls of  $\mathbb H^n$ and $S^n$ respectively, both centered at a fixed reference point  which is identified with the point $(0,\cdots, 0, 1)$ of ${\mathbb R}^{n+1}$ through the projective models.  
Then we define the spherical Hilbert metric as follows:

\begin{definition} \label{d:hilbert-s}
For a pair of points $x$ and $y$ in $B_\rho^s$ with $0< \rho \leq \pi/2$, 
the Hilbert distance from $x$ to $y$ is defined by 
\begin{equation*}
H^s_\rho(x, y)= 
\begin{cases}  \displaystyle \log \frac{\sin d(x, b(x,y))}{\sin d(y,  b(x,y))} \cdot  \frac{\sin d(y, b(y,x))}{\sin d(x,  b(y,x))} & \text{if $x\not=y$, }
\\
0 &  \text{otherwise}
\end{cases}
\end{equation*}
where $b(x, y)$ is the point where the  geodesic ray $R(x, y)$ from $x$ through $y$ hits the boundary  $\partial B^s_\rho$ of $B^s_\rho$. 
\end{definition}

\noindent  We then define hyperbolic Hilbert metric:

\begin{definition} \label{d:hilbert-h}
For a pair of points $x$ and $y$ in $B_\rho^h$ with $\rho >0$, 
the Hilbert distance from $x$ to $y$ is defined by 
\begin{equation*}
H^h_\rho(x, y)= 
\begin{cases}  \displaystyle \log \frac{\sinh d(x, b(x,y))}{\sinh d(y,  b(x,y))} \cdot  \frac{\sinh d(y, b(y,x))}{\sinh d(x,  b(y,x))} & \text{if $x\not=y$, }
\\
0 &  \text{otherwise}
\end{cases}
\end{equation*}
where $b(x, y)$ is the point where the  geodesic ray $R(x, y)$ from $x$ through $y$ hits the boundary  $\partial B^h_\rho$ of $B^h_\rho$. 
\end{definition}

Through the projective maps $P_s$ and $P_h$,  each geodesic ball in the curved spaces corresponds to a Euclidean ball. From Theorems \ref{th:s} and \ref{th:h}, the Hilbert distance, which is the logarithm of the spherical/hyperbolic cross ratio of the  quadruple $ (x, y, b(x,y), b(y,x))$ defined on the geodesic ball, is preserved by projective maps; that is, $P_h$ and $P_s$ are isometries of the Hilbert metrics.  Hence it follows that in the spherical case, we have the following:

\begin{corollary}
The geodesic ball $B^s_\rho$ in the unit sphere $S^n$  with $0 < \rho < \pi/2$ with its spherical Hilbert metric $H^s_\rho$ is isometric to the hyperbolic space ${\mathbb H}^n$. 
\end{corollary}

And in the hyperbolic case, we have:
\begin{corollary}
The geodesic ball $B^h_\rho$ in the hyperbolic space  ${\mathbb H}^n$ with its hyperbolic Hilbert metric $H^h_\rho$ is isometric to the hyperbolic space ${\mathbb H}^n$. 
\end{corollary}

We can consider the spaces $(B^s_\rho, H^s_\rho)$ and $(B^h_\rho, H^h_\rho)$ as {\it generalized Beltrami-Klein models of the hyperbolic space.}

\noindent  Athanase Papadopoulos,  Institut de Recherche Math\'ematique Avanc\'ee,
Universit{\'e} de Strasbourg and CNRS,
7 rue Ren\'e Descartes,
 67084 Strasbourg Cedex, France;  email : athanase.papadopoulos@math.unistra.fr

 \bigskip
 
 \noindent Sumio Yamada, Mathematical Institute, Tohoku University, Sendai, 980-8578, Japan; email : yamada@math.tohoku.ac.jp

\end{document}